\documentclass[12pt]{amsart}

\usepackage{amsfonts,amsmath,amssymb,amsthm,graphicx,enumerate, mathabx}
\usepackage[english]{babel}
\usepackage[utf8]{inputenc}
\usepackage[T1]{fontenc}

\usepackage{framed}

\usepackage{marginnote,color}

\usepackage{array}
\usepackage{hyperref}

\usepackage[margin=2.8cm]{geometry}

\newcommand{\R}{\mathbb{R}}

\newcommand{\E}{\mathbb{E}}
\renewcommand{\H}{\mathbb{H}}

\renewcommand{\P}{\mathbb{P}}

\newcommand{\N}{\mathbb{N}}
\newcommand{\Z}{\mathbb{Z}}
\newcommand{\indic}{{\bf 1}}
\newcommand{\ind}[1]{{\indic_{\{#1\}}}}

\newcommand{\epst}{{\widetilde\varepsilon}}

\newcommand{\Fr}{\mathcal F}

\newcommand{\eps}{\varepsilon}

\newcommand{\dm}{\begin{pmatrix}} 
\newcommand{\fm}{\end{pmatrix}}
\newcommand{\ddm}{\begin{vmatrix}} 
\newcommand{\fdm}{\end{vmatrix}}

\usepackage{textcomp,eurosym}

\newcommand{\s}{\,|\,} 

\newcommand{\st}{\,:\,} 

\newtheorem{lemma}{Lemma}
\newtheorem{proposition}{Proposition}

\newtheorem*{remark*}{Remark}

\author[V. Sidoravicius]{Vladas Sidoravicius}
\address{Courant Institute of Mathematical Sciences, New York\\
NYU-ECNU Institute of Mathematical Sciences at NYU Shanghai\\
Cemaden, São José dos Campos.}
\email{vs1138@nyu.edu}

\author[L. Tournier]{Laurent Tournier}
\address{LAGA, Universit\'e Paris 13, Sorbonne Paris Cit\'e, CNRS, UMR 7539, 93430 Villetaneuse, France. 
 }
\email{tournier@math.univ-paris13.fr}

\title[]{Note on a one-dimensional system of annihilating particles}

\begin{document}

\ 
\vspace{-2cm}
\maketitle

{\footnotesize \noindent{\slshape\bfseries Abstract.} 
We consider a system of annihilating particles where particles start from the points 
of a Poisson process on either the full-line or positive half-line and move at constant 
i.i.d.\ speeds until collision. When two particles collide, they annihilate. We assume 
the law of speeds to be symmetric. We prove almost sure annihilation of positive-speed 
particles started from the positive half-line, and existence of a regime of survival of 
zero-speed particles on the full-line in the case when speeds can only take 3 values. 
We also state open questions. 
}

\section{Introduction}
   \label{s:intro}

Let us first define informally the model that we are working on. Particles are released from the locations of a Poisson point process on either the full-line $\R$ or the half-line $\R_+$ with i.i.d.\ velocities sampled from a distribution $\mu$ with bounded support. Each particle moves at constant velocity, and when two particles collide, they annihilate. 
We are interested in the possible survival of some particles forever. 

In this note, we treat the case of a symmetric distribution $\mu$ (i.e.\ $v\stackrel{\rm (d)}=-v$ if $v$ has law $\mu$) for particles starting from the half-line, and the case of a symmetric distribution $\mu$ on $\{-1,0,1\}$ for particles starting from the full-line. 
For symmetric distributions $\mu$, and particles starting on the full-line, only 0-speed particles may survive. The question on the half-line is more intriguing. In Section~\ref{sec:symmetric}, we present an argument based on symmetry to conclude that positive-velocity particles annihilate almost surely. We believe, but could not prove, that negative-velocity particles have positive chance to survive.
On the full-line, we are specifically interested in the symmetric discrete velocity case $\mu=\frac{1-p}2\delta_{-1}+p\delta_0+\frac{1-p}2\delta_1$, where $p$ is a parameter in $[0,1]$. In this case, numerical evidence and predictions from physicists (see below) suggest that there is a phase transition: for~$p$ small (approximately for $p<0.25$), every particle annihilates almost surely but every site is still crossed by some particles at arbitrarily large times, while for larger values of $p$, 0-speed particles manage to survive with positive probability. In Section~\ref{sec:3speeds}, we present a simple proof of survival at $p>1/3$ and discuss some extensions. We could not prove that particles die at small $p$, and leave it as another open question (see the discussion below regarding this question).

Although this particle system was little known by mathematicians until recently, it turns out to have been introduced in the physics community in the 1990's (cf.~\cite{ben1993decay}), where it is known as \emph{ballistic annihilation}, and has been an active subject of research at the time. Note however that the special case of two speeds was considered earlier (\cite{elskens1985annihilation,krug1988universality}), cf.\ also, for a mathematical treatment~\cite{belitsky1995ballistic}. In the case of a continuous speed distribution (with Poisson initial distribution on the full-line), several heuristics and numerical simulations suggest a polynomial decay of the density of particles and of their average speed, and universal relations between their exponents (cf.~\cite{ben1993decay,krapivsky2001ballistic, trizac2002kinetics}). For discrete velocity distributions, beyond the case of two speeds, the cases of three speeds (as in the present paper) and of four speeds have been considered. In the symmetric three-speed case $\mu=\frac{1-p}2\delta_{-1}+p\delta_0+\frac{1-p}2\delta_1$, Krapivsky, Redner and Leyvraz~\cite{krapivsky1995ballistic} infered the critical value $p=\frac14$ from heuristic and numerical arguments, together with exponents for the decay of density $c_v(t)$ of each speed at time~$t$: at $p<\frac14$, one should have $c_0(t)\approx t^{-1}$ and $c_{+1}(t)\approx t^{-1/2}$; at $p=\frac14$, $c_0(t)\approx c_{+1}(t)\approx t^{-2/3}$; and at $p>\frac14$, $c_0(t)\to C$ and $c_\pm(t)$ should decay faster than polynomially. Essentially at the same time, Droz, Frachebourg, Piasecki and Rey~\cite{droz1995ballistic}, using computation of Piasecki~\cite{piasecki1995ballistic}, strongly supported these asymptotics by showing that exact computations can be done in this and possibly other discrete cases, furthermore giving explicit prefactors: 
\begin{align*}
	\text{if }p<\frac14,\ & 
		c_0(t)\sim\frac{2p}{(1-4p)\pi}t^{-1}\text{ and }c_{+1}(t)\sim\sqrt{\frac1{\pi}\Big(\frac14-p\Big)}\cdot t^{-1}\\
	\text{if }p=\frac14,\ &   
		c_0(t)\sim\frac{2^{2/3}}{4\Gamma(2/3)^2}t^{-2/3}\text{ and }c_{+1}(t)\sim\Big(\frac{2^{2/3}}{8\Gamma(2/3)^2}+\frac3{8\Gamma(1/3)}\Big)t^{-2/3}\\
	\text{if }p>\frac14,\ &
		c_0(t)\to 2-\frac1{\sqrt p}\text{ and }c_{+1}(t)\sim At^{-3/2}e^{-cut},
\end{align*}
where $A$ and $u$ are also explicit functions of $p$. These very precise results come from the exact resolution of an involved differential equation satisfied by the probability density function of the interdistance between neighbor particles at time $t$ conditioned on having given speeds. The derivation and resolution of this equation is not entirely written in a rigorous way but it is not unlikely it might be turned to a formal proof. Unfortunately though, as Krapivsky et al.~\cite{krapivsky1995ballistic} already pointed out, the previous explicit resolution is ``a formidable task'' and there is still need for methods ``that would provide better intuitive insights into the intriguing qualitative features of ballistic annihilation''. In particular, this computation, even if indeed formally correct, does not seem to give an intuitive proof of survival for any small a value of $p$. 

On a side note, coalescing processes also have interesting ballistic counterparts, cf. for instance the model from~\cite{ermakov1998some}, where speeds are resampled at collision, or the physically relevant model of ballistic aggregation, where particles have a speed and mass, and coalescence occurs with conservation of mass and momentum (i.e.\ particles with mass and speed respectively given by $(m_1,v_1)$ and $(m_2,v_2)$ produce $(m_1+m_2,(m_1v_1+m_2v_2)/(m_1+m_2))$), cf.~\cite{carnevale1990statistics}, \cite{martin1994one}. An account of related models can be found in the review~\cite{redner1997scaling}. 

The interest for ballistic annihilating particles has known a recent revival, among the mathematical community, after a closely related ``bullet problem'' appeared as a challenge by LTK software engineer David Wilson on IBM website~\cite{ibm}. In the ``bullet'' model, particles are released from 0 at times of a Poisson point process, with i.i.d.\ positive velocities $(w_k)_{k\ge1}$, and annihilate on collision. It was conjectured that, if the law of speeds is uniform on $[0,1]$,  there is a velocity $v_c\in(0,1)$ such that bullets slower than $v_c$ annihilate almost surely, while faster ones may survive. This question was apparently not considered in physics papers. Compared to ballistic annihilation started on $\R_+$, the ``bullet'' model corresponds to switching time with space, which results in replacing speeds $w_k$ by $v_k=1/w_k$. Our results on the half-line can therefore be rephrased for this other model. Note that the present version of the model enjoys invariance by linear transforms of speeds and symmetry properties that make it more suitable for study than the ``bullet'' version. 

We learned before publishing that other authors~\cite{junge} independently obtained results similar to ours. The physics literature on ballistic annihilation mentioned above, in particular Reference~\cite{droz1995ballistic}, was pointed to us after we finished the first version of this paper. 

\begin{figure}
 \begin{center}
  \includegraphics[width=15cm]{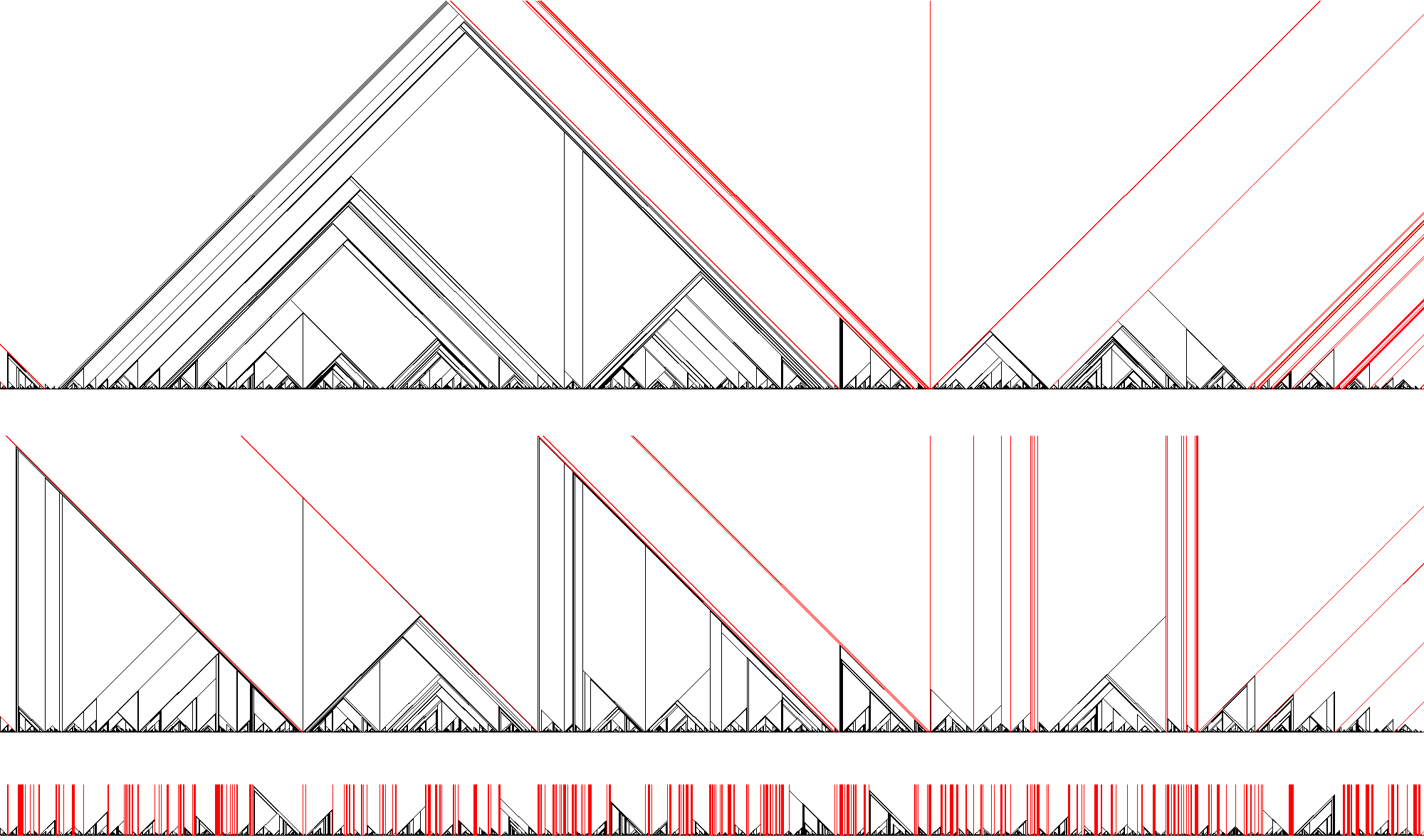}
 \caption{ \label{fig:p025}	Simulations for $\mu=\frac{1-p}2\delta_{-1}+p\delta_0+\frac{1-p}2\delta_1$, with $p=0.24$ (top), $p=0.25$ (middle) and $p=0.26$ (bottom), with the horizontal direction as space and vertical (up) as time. The bottom picture contains 10 million particles. Note that, as can be guessed, the pictures are coupled by removing some 0-speed particles from the lower configuration, hence pictures have slightly different horizontal scale (or Poisson intensity).}
 \end{center}
 \end{figure}

\begin{figure}
 \begin{center}
  \includegraphics[width=12cm]{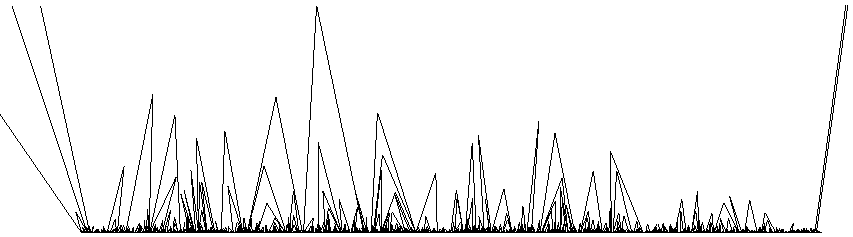}
 \caption{ \label{fig:symmetric}	Simulation for $\mu$ uniform on $[-1,1]$ and $10^5$ particles. }
\end{center}
 \end{figure}

\section{Definitions and basic properties}\label{sec:definitions}

Let $(x_k)_{k\in\Z}$ be a Poisson point process with intensity 1 on $\R$ under Palm measure, i.e.\ conditioned on containing 0: 
\[\cdots<x_{-1}<x_0=0< x_1<\ldots,\] which is equivalent to $(x_k)_{k\ge1}$ and $(-x_{-k})_{k\ge1}$ being independent Poisson point processes on $\R_+$. The points $x_k$ are meant as starting locations of particles.

Let $\mu$ be a distribution on $\R$ with bounded support. Let $(v_k)_{k\in\Z}$ be i.i.d.\ random variables with law $\mu$, independent of $(x_k)_{k\in\Z}$, standing for the velocities of the particles. 

Let us denote by $\P_\mu$ the law of $(x_k)_k,(v_k)_k$. We shall write $\P_p$ in the particular case when $\mu=\frac{1-p}2\delta_{-1}+p\delta_0+\frac{1-p}2\delta_1$, for $p\in[0,1]$.

The process is defined as follows. From each location $x_k$, a particle is released at time~0, with speed $v_k$ (hence going to the right if $v_k>0$, to the left otherwise). Particles move at constant speed until they collide with another particle, at which point both annihilate and therefore don't take part in later collisions. Due to interdistances between starting locations having an atomless distribution and being independent of speeds, almost surely no triple collision happens. The model is obviously well-defined if only finitely many particles are considered, for there is a chronologically first collision to be dealt with. In the case when particles initially lie on $\R_+$, well-definedness is still ensured if speeds are lower bounded: for $i,j\in\Z_+$, given $x_i,x_j,v_i,v_j$, the collision of $i$ and $j$ can indeed be checked by only considering the system formed by the finitely many particles on the left of the rightmost location where a triple collision with $i$ and $j$ could be triggered. Using symmetry, collisions are then well-defined on $\R$ since speeds are assumed bounded.

\newcommand{\hit}[1]{\overset{}{\underset{#1}{\longmapsto}}}
\newcommand{\hits}[1]{\overset{}{\underset{#1}{\longleftrightarrow}}}

For $k,l\in\Z$ belonging to an interval $I\subset\R$, we denote by $k\hits I l$ the property that the particles released from $x_k$ and $x_l$ mutually annihilate in the system restricted to particles departing from $\{x_i\st i\in I\}$. We denote by $k\hits I\infty$ the property that the particle from $x_k$ survives, i.e.\ is not annihilated by another particle, in this restricted system.

Note that, for given starting locations, collision events $\{i\hits I j\}$ only depend on ratios of differences of speeds, so that adding the same amount to every speed (i.e.\ shifting $\mu$), or multiplying them by a constant, does not change the probability of any event relating to collisions. We shall refer to this simple fact as the \emph{linear speed-change invariance property}. In particular, several ``bullet'' models (see Introduction) are in correspondence with a given $\mu$, through $w_k\mapsto v_k=1/(\alpha+\beta w_k)$ for any $\alpha,\beta$ such that $\alpha+\beta w_k> 0$ a.s.. 

Also, a symmetry property obviously holds with respect to the reflection $x\mapsto -x$, namely that the law of the model with particles starting from $I\subset\R$ is the same as that of the reflection of the model with particles starting from $-I$ with speeds sampled from $\mu(-{\rm d}v)$.

\section{Extinction for symmetric distributions}\label{sec:symmetric}

\begin{proposition}\label{pro:symmetry}
Assume $\mu$ is symmetric, $\mu\ne\delta_0$. 
\begin{enumerate}
[a)]
	\item For any $v>0$, $\P_\mu(0\hits{\R_+}\infty\s v_0=v)=0$, where the conditioning here is understood as setting $v_0=v$ and letting $(v_k)_{k\ne0}$ be i.i.d.\ with law $\mu$.
	\item If $\P_\mu(0\hits\R\infty)=0$, then almost surely infinitely many particles cross~0 in the system restricted to particles starting from $\R_+$. 
\end{enumerate}
\end{proposition}

This Proposition will follow from the two lemmas below. 

\begin{lemma}\label{lem:hulls}
For $x\in\R$, let $N_x=\#\{(i,j)\in\Z^2\s i<x<j,\ i\hits\R j\}$ denote the number of couples of particles that meet over $x$. 
If $\P_\mu(0\hits\R\infty)=0$, then a.s., for all $x$, $N_x=\infty$.
\end{lemma}

\begin{proof}[Proof of Lemma~\ref{lem:hulls}]
We use a parity argument inspired by~\cite[page 43]{aizenman-nachtergaele}. First note that the event $\{N_x=\infty\}$ does not depend on $x$, is translation invariant, and therefore has probability 0 or 1 due to ergodicity. Assume that $\P_\mu(0\hits\R\infty)=0$ and, by contradiction, that $N_{x_1/2}<\infty$ almost surely. Let us view the trajectories of particles as space-time curves (actually, line segments) in the upper half-plane $\H$ (cf.\ Figures~\ref{fig:p025} and~\ref{fig:symmetric}). Since $\P_\mu(0\hits\R\infty)=0$, almost surely for every $i\in\Z$, there is $j\in\Z$ such that $i\hits\R j$, and the trajectories of $i$ and $j$ split $\H$ into one finite and one infinite components. Note also that all these pairs of trajectories are disjoint of each other. Since furthermore $N_{x_1/2}<\infty$ a.s., the union of these curves delimitates exactly one infinite component in $\H$. This implies that the event $\{N_{x_1/2}\text{ is even}\}$ is invariant by even shifts, i.e.\ by replacing $(x_k,v_k)_{k\in\Z}$ by $(x_{k+2}-x_2,v_{k+2})_{k\in\Z}$. $N_{x_1/2}$ is indeed the number of curves that separate $\frac{x_1}2$ from the infinite component, and a shift amounts to crossing a boundary and thus increases or decreases by 1 the number of curves separating from the infinite component, hence changing parity. By ergodicity, this event therefore has to have probability~0 or~1. However, its probability is $1/2$ due to shift invariance and alternance of parities. The conclusion follows from this contradiction.
\end{proof}

\begin{lemma}\label{lem:only0}
Assume $\mu$ is symmetric, $\mu\ne\delta_0$. Then $\P_\mu(0\hits{\R}\infty\s v_0\ne 0)=0$. 
\end{lemma}

\begin{proof}[Proof of Lemma~\ref{lem:only0}]
If $\P_\mu(v_0>0,\,0\hits\R\infty)>0$ then, by ergodicity, almost surely a positive density of particles survive and have a positive speed, and by symmetry the same holds for negative speeds, contradicting their mutual survival. 
\end{proof}

\begin{proof}[Proof of Proposition~\ref{pro:symmetry}]
Let us first prove a). 
If $\P_\mu(0\hits\R\infty)>0$, then by Lemma~\ref{lem:only0} surviving particles have 0 speed, and by ergodicity a positive density of them survive, which prevents any positive speed particle from surviving. We can thus now suppose that $\P_\mu(0\hits\R\infty)=0$.
Assume by contradiction that there is $v>0$ such that $\P_\mu(0\hits{\R_+}\infty\s v_0=v)>0$. Since this probability is non-increasing in $v$, we may furthermore choose $v$ small enough so that $\mu((v,+\infty))>0$. 
We have by symmetry that $\P_\mu(0\hits{\R_-}-\infty\s v_0=-v)>0$ and thus, by independence, with positive probability we may have (cf.~Figure~\ref{fig:proof}) that, at the same time, $v_0>v$, $v_1<0$, a $(-v)$-speed particle at $0$ would survive $\R_-$ and a $v$-speed particle at $x_1$ would survive $\R_+$. On this event, both an additional $v$-speed particle and an additional $(-v)$-speed particle launched at $0$ (at time $0^+$) would survive, counting with the particle already at 0. Since this has positive probability, by ergodicity we conclude that almost surely there is a positive density of Poisson points where this happens. This implies that $N_0<\infty$, for if $j,k$ are indices of such points, with $j<0<k$, then the surviving trajectories of a $(-v)$ particle at $k$ and of a $v$-particle at $j$ bound the couples straddling 0 to lie between $j$ and $k$ and thus be finitely few (cf.\ Figure~\ref{fig:proof}). Lemma~\ref{lem:hulls} yields a contradiction. 

\begin{figure}
 \begin{center}
  \includegraphics[width=6cm]{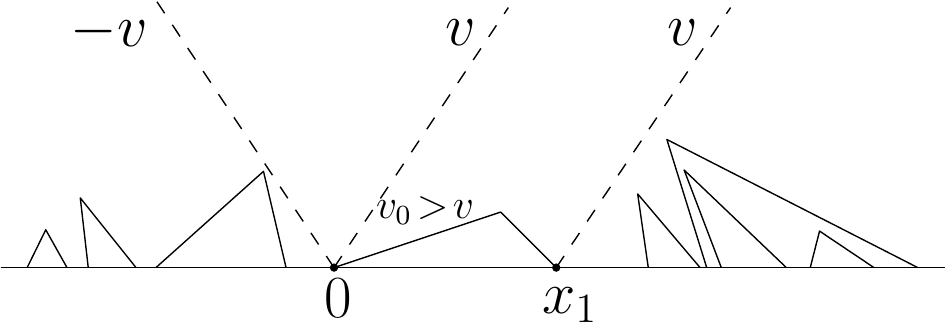}
  \includegraphics[width=6cm]{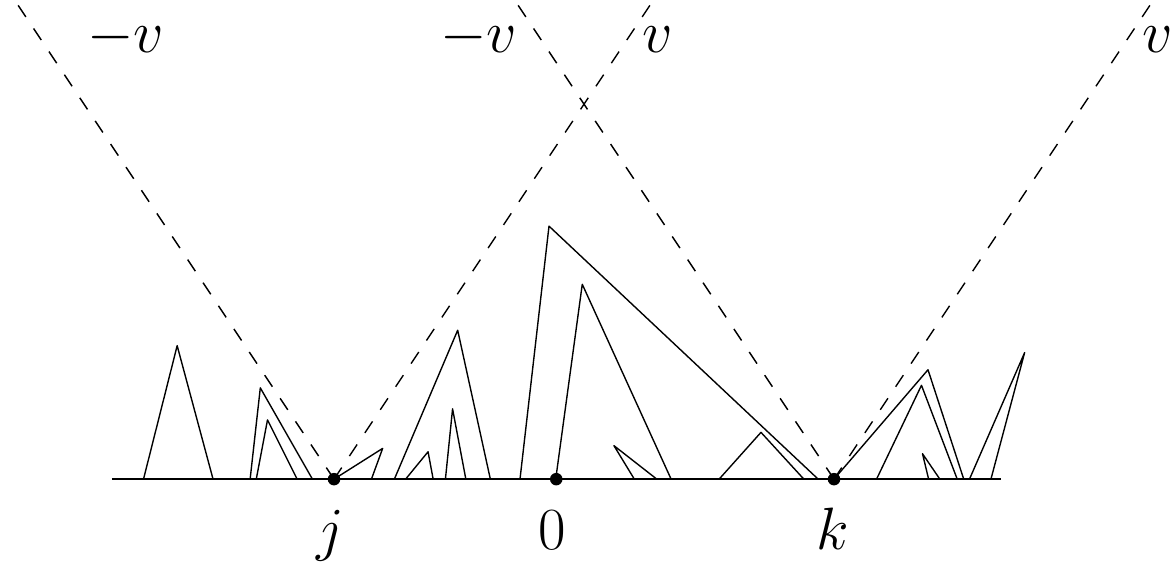}\\
  \includegraphics[width=6cm]{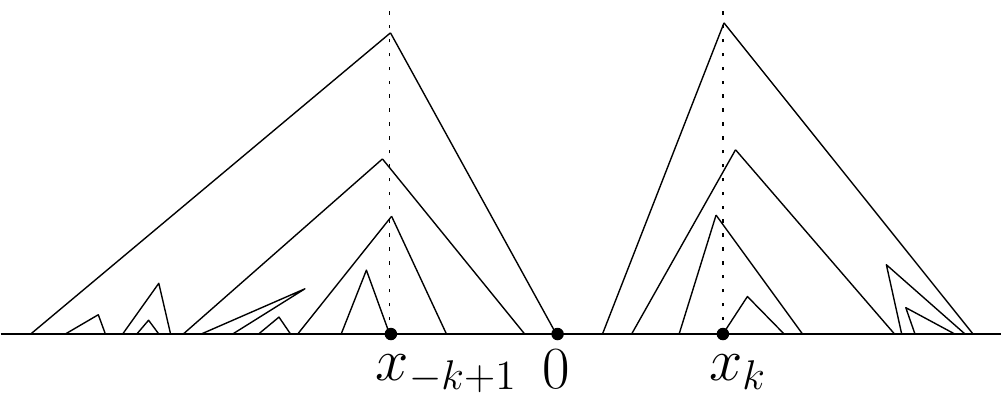}
 \caption{ \label{fig:proof}Illustrations of the proof of Proposition~\ref{pro:symmetry} a) (top) and b) (bottom). }
\end{center}
 \end{figure}

Let us now consider b).  Suppose $\P_\mu(0\hits\R\infty)=0$. Let us denote by $N_+$ (resp.\ $N_-$) the number of particles that ever cross $0$ in the system restricted to particles on $\R_+$ (resp.\ $\R_-$). Assume by contradiction that $\P_\mu(N_+<\infty)>0$. One may notice that $N_+=N_-<\infty$ doesn't automatically imply that $N_0<\infty$. However, there is $k\in\N$ and rational locations $u_1<v_1<u_2<v_2<\cdots<u_k<v_k$ such that, with positive probability, $N_+=k$ and these $k$ particles cross 0 at times $\tau_1\in[u_1,v_1]$,..., $\tau_k\in[u_k,v_k]$. By symmetry, the same event relative to $\R_-$ has the same positive probability. We can then produce a positive probability event such that  (cf.\ Figure~\ref{fig:proof}) the above happens for the process on the right of $x_k$ and on the left of $x_{-k+1}$ respectively, and such that the particles from 1 to $k$ (resp.\ from $0$ to $-k+1$) meet the $k$ particles that arrive from the right (resp.\ from the left). However we then have $N_0=0$ on this event, contradicting Lemma~\ref{lem:hulls}. 
\end{proof}

Let us note that the following weaker result, namely the extinction of a positive-speed particle at 0 when the law of its speed is $\mu$, is easier:

\begin{proposition}\label{pro:symmetry}
Assume $\mu$ is symmetric $\mu\ne\delta_0$. Then $\P_\mu(0\hits{\R_+}\infty\s v_0>0)=0$. 
\end{proposition}

\begin{proof}[Proof of the proposition]
By contradiction, assume $\P_\mu(0\hits{\R_+}\infty\s v_0>0)>0$. 
Then we may find $v>0$ such that
\begin{equation}\label{eqn:defv}
\P_\mu(v_0\in(0,v])>0\qquad\text{and}\qquad \P_\mu(0\hits{\R_+}\infty\s v_0=v)>0, 
\end{equation}
where the conditioning here is merely understood as letting $v_0=v$ and letting $(v_k)_{k\ne0}$ be i.i.d.\ with law $\mu$. Indeed, one may for instance take $v$ to be the median or any quantile of $\P_\mu(v_0\in\cdot\s 0\hits{\R_+}\infty,\, v_0>0)$, for then $\P_\mu(0<v_0\le v)>0$ and $\P_\mu(v_0\ge v,\ 0\hits{\R_+}\infty)>0$, and the last probability is smaller than $\P_\mu(0\hits{\R_+}\infty\s v_0=v)\mu((v,\infty))$ because replacing the speed $v_0$ at $x_0$ by $v\le v_0$ preserves survival on $\R_+$. 

Then~\eqref{eqn:defv}, with its symmetric $\P_\mu(0\hits{\R_-}\infty\s v_0=-v)>0$, yields, for any $w\in[-v,v]$, 
\begin{align*}
\P_\mu(0\hits{\R}\infty\s v_0=w)
	& =\P_\mu(0\hits{\R_+}\infty\s v_0=w)\P_\mu(0\hits{\R_-}\infty\s v_0=w)\\
	& \ge \P_\mu(0\hits{\R_+}\infty\s v_0=v)\P_\mu(0\hits{\R_-}\infty\s v_0=-v)>0.
\end{align*}
Hence in particular, since $\mu((0,v])>0$ by\eqref{eqn:defv},
\[\P_\mu(0\hits{\R}\infty, v_0\in(0,v])>0,\]
which contradicts Lemma~\ref{lem:only0}. This proves the proposition.
\end{proof}

We may remark that Lemma~\ref{lem:only0} ensures that the assumption of Proposition~\ref{pro:symmetry} b) is satisfied in particular if $\mu(\{0\})=0$. Also, the arguments of the proof of this lemma show that surviving particles on $\R$ have the same speed, which is deterministic and therefore has to be an atom of $\mu$: 

\begin{lemma}
If $\P_\mu(0\hits\R\infty)>0$, then there is $v\in\R$ such that $\P_\mu(v_0=v\s 0\hits\R\infty)=1$. 
\end{lemma}

\begin{proof}
Assume $\P_\mu(0\hits\R\infty)>0$. Let $\nu$ denote the law of $v_0$ given $\{0\hits\R\infty\}$. Let $w\in\R$. If $\P_\mu(v_0>w,\,0\hits\R\infty)>0$ and $\P_\mu(v_0<w,\,0\hits\R\infty)>0$, then, by ergodicity, almost surely a positive density of particles survive and have a speed $>w$, and similarly with speed $<w$, which is contradicting their mutual survival. Therefore, either $\nu((w,+\infty))=0$ or $\nu((-\infty,w))=0$. Hence $\nu$ has to be a Dirac measure, for otherwise taking $w$ to be its median yields a contradiction.
\end{proof}

\section{Survival for 3-speeds distributions}\label{sec:3speeds}

Recall that $\P_p$ refers to $\P_\mu$ where $\mu=\frac{1-p}2\delta_{-1}+p\delta_0+\frac{1-p}2\delta_1$. 

\begin{proposition}
Assume $p>\frac13$. Then $\P_p(0\hits\R\infty)>0$.
\end{proposition}

Let us recall the following fact, that is a particular case of Proposition~\ref{pro:symmetry} (or of Lemma~\ref{lem:only0}, since $(+1)$-particles can only be hit by particles starting on their right).

\begin{lemma}\label{lem:1dies}
Assume $p>0$. Then $\P_p(0\hits{\R_+}\infty\s v_0=+1)=0$. 
\end{lemma}

\begin{proof}[Proof of the proposition]
The proof procedes by a definition of an exploration of the configuration on the right of~$0$. This exploration will produce a sequence of random locations $K_0,K_1,\ldots\in\Z$ along with random signs $\eps_0,\ldots$, and $\epst_0,\ldots$. The locations $K_n$ will be predictable stopping locations, in the sense that $\{K_n=k\}\in\Fr_{k-1}:=\sigma((x_j,v_j)\,;\,j=0,\ldots,k-1)$ for all $k\in\N$. And the signs $\eps_n$ will account for the number of surviving particles from $x_0$ to $x_{K_n}$ with speed either 0 or -1, cf.~\eqref{eqn:epsilon}. The signs $\epst_n$ are introduced for a technical reason explained later. 

Let $K_0=0$. Then, for $n\ge0$, given $K_n$, let us define $K_{n+1}$, $\eps_n$ and $\epst_n$ as follows:
\begin{itemize}
	\item if $v_{K_n}=0$, then $\eps_n=\epst_n=+1$ and $K_{n+1}=K_n+1$;
	\item if $v_{K_n}=+1$, then $\eps_n=\epst_n=0$, and Lemma~\ref{lem:1dies} ensures that there is $k>K_n$, such that $K_n\hits{[K_n,k]}k$. Let $k'$ be the least such $k$, and $K_{n+1}=k'+1$;
	\item if $v_{K_n}=-1$, then 
	\begin{itemize}
		\item if the particle at $x_{K_n}$ reaches $0$, i.e.\ $K_n\hits{\R_+}\infty$, then $\eps_n=\epst_n=-1$ and $K_{n+1}=K_n+1$;
		\item else, there is $i\in[0,K_n)$ such that $i\hits{[0,K_n]}K_n$. If $i$ was surviving before $K_n$, i.e.\ if $i\hits{[0,K_n)}\infty$, then $\eps_n=\epst_n=-1$ and $K_{n+1}=K_n+1$. Else, this means there is $j\in[0,i)$ such that $j\hits{[0,K_n)}i$ (and we must have $v_i=0$ and $v_j=+1$), in which case $j\hits{[0,K_n]}\infty$ and Lemma~\ref{lem:1dies} provides $k>K_n$ such that $j\hits{[0,k]}k$. Let $k'$ be the least such $k$, and $K_{n+1}=k'+1$. In this last sub-case, we let $\eps_n=0$ and $\epst_n=-1$. 
	\end{itemize}
\end{itemize}
This construction yields, by induction, that, for all $n\ge1$, the sequence $K_1,\ldots,K_n$ contains all the locations of particles that survive in the system restricted to particles in $[0,K_{n+1})$, that none of these surviving ones has speed +1, and moreover that
\begin{align}\label{eqn:epsilon}
\sum_{0\le m\le n}\eps_m = &
\#\text{\Big(0-speed particle surviving in $[0,K_{n+1})$\Big)}\\ - & \#\text{\Big((-1)-speed particle surviving in $[0,K_{n+1})$\Big)} \notag
\end{align}
or, in other words, 
\[\sum_{0\le m\le n}\eps_m = \sum_{\substack{k\in[0,K_{n+1}),\\ k\hits{[0,K_{n+1})}\infty}}(\ind{v_k=0}-\ind{v_k=-1}).\]
In particular, if $v_0=0$, then $0\hits{[0,K_{n+1})}\infty$ if and only if $\eps_0+\cdots+\eps_m>0$ for $m=0,\ldots,n$. 
However, 
\begin{equation}\label{eqn:eps_tilde}
\text{for all $n\in\N$},\qquad\eps_n\ge\epst_n,
\end{equation}
and the sequence $(\epst_n)_n$ is i.i.d.\ with law $\frac{1-p}2\delta_{-1}+p\delta_1+\frac{1-p}2\delta_{0}$. Indeed, $\epst_n $ is a function of $v_{K_n}$ and $(v_{K_n})_n$ is i.i.d.\ with law $\mu$ due to the construction (recall that $(K_n)_{n\ge0}$ are predictable stopping locations). 

Assuming $p>\frac13$, we have $\E_p[\epst_0]>0$, and a straightforward corollary of the law of large numbers implies that, with positive probability, $\epst_0+\cdots+\epst_n>0$ for all $n$. Due to~\eqref{eqn:eps_tilde}, we also have $\eps_0+\cdots+\eps_n>0$ for all $n$ with positive probability, hence $\P_p(0\hits{\R_+}\infty\s v_0=0)>0$. Due to symmetry, and independence between positive and negative half-axes, the proposition follows. 
\end{proof}

\subsection*{Remarks}
\begin{itemize}
	\item The argument can be modified so as to yield survival below $1/3$. One can indeed consider an exploration of locations 3 by 3, and similarly list the 27 situations and associate similar values $\eps_n\in\{-3,\ldots,3\}$ to each of them, \emph{except} to the only case $(+1,0,-1)$, which depending on interdistances can, with equal probability, contribute to $\eps_n=-1$ (if +1 and 0 meet first) or $\eps_n=0$ (if 0 and -1 meet first), which introduces a drift at $p=1/3$. This argument gives survival with positive probability for $p>0.32803...$
	\item The symmetry assumption can be dropped, if only survival on $\R_+$ is considered. Namely, assuming $\mu=q\delta_{-1}+p\delta_0+r\delta_1$ with $p>q$, if we further assume that $(+1)$-speed particles annihilate a.s.\ on $\R_+$ (as in the conclusion of Lemma~\ref{lem:1dies}), then the proof carries over exactly and proves that $\P_\mu(0\hits{\R_+}\infty\s v_0=0)>0$; but if on the contrary a $(+1)$-speed particle at $x_0$ survives with positive probability on $\R_+$, then substituting it for a $0$-speed particle obviously preserves survival on $\R_+$ hence we still have $\P_\mu(0\hits{\R_+}\infty\s v_0=0)>0$.
	\item Due to linear speed-change invariance (cf.~Section~\ref{sec:definitions}), the previous remark implies that if $\mu=p\delta_\alpha+q\delta_\sigma+r\delta_\beta$ with $\alpha<\sigma<\beta$, then the assumption $p>q$ implies $\P_\mu(0\hits{\R_+}\infty\s v_0=\sigma)>0$.
	\item For any law $\mu$ such that $\mu(\{0\})>1/2$, the survival of some 0-speed particles on $\R$ is easy: for instance, a $0$-speed particle at 0 survives the particles from $\R_+$ on the positive-probability event that $\sum_{k=1}^n(\ind{v_k=0}-\ind{v_k\ne0})>0$ for all $n\ge1$. 
\end{itemize}

\section*{Acknowledgments}

The authors wish to thank Gady Kozma for stimulating discussions, in particular regarding reference~\cite{aizenman-nachtergaele}, and Sidney Redner for letting them know of the literature on the model. 

\bibliographystyle{acm}
\bibliography{biblio}

\begin{thebibliography}{10}

\bibitem{aizenman-nachtergaele}
{\sc Aizenman, M., and Nachtergaele, B.}
\newblock Geometric aspects of quantum spin states.
\newblock {\em Communications in Mathematical Physics 164}, 1 (1994), 17--63.

\bibitem{belitsky1995ballistic}
{\sc Belitsky, V., and Ferrari, P.~A.}
\newblock Ballistic annihilation and deterministic surface growth.
\newblock {\em Journal of statistical physics 80}, 3-4 (1995), 517--543.

\bibitem{ben1993decay}
{\sc Ben-Naim, E., Redner, S., and Leyvraz, F.}
\newblock Decay kinetics of ballistic annihilation.
\newblock {\em Physical review letters 70}, 12 (1993), 1890.

\bibitem{carnevale1990statistics}
{\sc Carnevale, G., Pomeau, Y., and Young, W.}
\newblock Statistics of ballistic agglomeration.
\newblock {\em Physical review letters 64}, 24 (1990), 2913.

\bibitem{droz1995ballistic}
{\sc Droz, M., Rey, P.-A., Frachebourg, L., and Piasecki, J.}
\newblock Ballistic-annihilation kinetics for a multivelocity one-dimensional
  ideal gas.
\newblock {\em Physical Review E 51}, 6 (1995), 5541.

\bibitem{junge}
{\sc Dygert, B., Junge, M., Kinzel, C., Raymond, A., Slivken, E., and Zhu, J.}
\newblock The bullet problem with discrete speeds.
\newblock {\em arXiv preprint arXiv:1610.00282\/} (2016).

\bibitem{elskens1985annihilation}
{\sc Elskens, Y., and Frisch, H.~L.}
\newblock Annihilation kinetics in the one-dimensional ideal gas.
\newblock {\em Physical Review A 31}, 6 (1985), 3812.

\bibitem{ermakov1998some}
{\sc Ermakov, A., T{\'o}th, B., and Werner, W.}
\newblock On some annihilating and coalescing systems.
\newblock {\em Journal of statistical physics 91}, 5 (1998), 845--870.

\bibitem{ibm}
{\sc Kleber, M., and Wilson, D.}
\newblock ``{P}onder {T}his'' {IBM} research challenge.
\newblock
  \url{https://www.research.ibm.com/haifa/ponderthis/challenges/May2014.html},
  May 2014.

\bibitem{krapivsky1995ballistic}
{\sc Krapivsky, P., Redner, S., and Leyvraz, F.}
\newblock Ballistic annihilation kinetics: The case of discrete velocity
  distributions.
\newblock {\em Physical Review E 51}, 5 (1995), 3977.

\bibitem{krapivsky2001ballistic}
{\sc Krapivsky, P., and Sire, C.}
\newblock Ballistic annihilation with continuous isotropic initial velocity
  distribution.
\newblock {\em Physical review letters 86}, 12 (2001), 2494.

\bibitem{krug1988universality}
{\sc Krug, J., and Spohn, H.}
\newblock Universality classes for deterministic surface growth.
\newblock {\em Physical Review A 38}, 8 (1988), 4271.

\bibitem{martin1994one}
{\sc Martin, P.~A., and Piasecki, J.}
\newblock One-dimensional ballistic aggregation: Rigorous long-time estimates.
\newblock {\em Journal of statistical physics 76}, 1 (1994), 447--476.

\bibitem{piasecki1995ballistic}
{\sc Piasecki, J.}
\newblock Ballistic annihilation in a one-dimensional fluid.
\newblock {\em Physical Review E 51}, 6 (1995), 5535.

\bibitem{redner1997scaling}
{\sc Redner, S.}
\newblock Scaling theories of diffusion-controlled and ballistically controlled
  bimolecular reactions.
\newblock {\em Nonequilibrium statistical mechanics in one dimension 1\/}
  (1997).

\bibitem{trizac2002kinetics}
{\sc Trizac, E.}
\newblock Kinetics and scaling in ballistic annihilation.
\newblock {\em Physical review letters 88}, 16 (2002), 160601.

\end{thebibliography}

\end{document}